\documentclass[10pt]{article}
\usepackage[a4paper,margin=24mm]{geometry}
\usepackage{amsmath,amssymb,amsthm,mathtools}
\usepackage[authoryear]{natbib}
\usepackage[hidelinks]{hyperref}
\usepackage{tikz-cd}
\newtheorem{theorem}{Theorem}

\newtheorem{example}[theorem]{Example}
\theoremstyle{remark}
\newcommand{\Ob}{\operatorname{Ob}}

\title{A Counterexample to the Open Question on Object Ideals}
\author{Qikai Wang, Yuxiao Wang and Haiyan Zhu\thanks{Corresponding author\\  Supported by the National Natural Science Foundation of China(12271481).}}
\date{}
\begin{document}
\maketitle
\begin{abstract} 
We give a counterexample to completeness descent from ideal cotorsion pairs to their objects.  
A radical-square-zero algebra on the two-cycle gives a finite-dimensional example.  The construction is intrinsically non-weakly-idempotent-complete.
\end{abstract}
 
Let \((\mathcal A;\mathcal E)\) be an exact category. 
For a full subcategory \(\mathcal C\subseteq\mathcal A\), put
\[
 {}^{\perp}\mathcal C
 =\{F\in\mathcal A\mid
     \operatorname{Ext}(F,C)=0\text{ for every }C\in\mathcal C\},
\]
and define \(\mathcal F^{\perp}\) dually for a full subcategory \(\mathcal F\subseteq\mathcal A\). A pair \((\mathcal F,\mathcal C)\)
of full subcategories of \(\mathcal A\) is a \emph{cotorsion pair} if $\mathcal F={}^{\perp}\mathcal C$, $\mathcal C=\mathcal F^{\perp}$.
Cotorsion pairs were introduced by \citep{Salce1979} for abelian groups.
Such a cotorsion pair is \emph{special precovering} if every object \(A\in\mathcal A\) occurs in a conflation $C\longrightarrow F\longrightarrow A$ with \(F\in\mathcal F\) and \(C\in\mathcal C\). 
It is \emph{special preenveloping} if every \(A\in\mathcal A\) occurs in a conflation $A\longrightarrow C'\longrightarrow F'$ with \(F'\in\mathcal F\) and \(C'\in\mathcal C\). 
It is \emph{complete} if both conditions hold.

If \(\mathcal X\) is a full additive subcategory of \(\mathcal A\), let \(\mathcal I(\mathcal X)\) denote the ideal of morphisms that factor through an object of \(\mathcal X\).  
For an ideal \(\mathcal I\), put
\[
  \operatorname{Ob}(\mathcal I)
  =\{X\in\mathcal A\mid 1_X\in\mathcal I\}.
\]
Following \citep{FGHT2013}, an ideal \(\mathcal I\) is called an \emph{object ideal} if it is generated by its objects, that is,
$\mathcal I=\mathcal I(\operatorname{Ob}(\mathcal I)).$
Equivalently, \(\mathcal I=\mathcal I(\mathcal X)\) for some full additive subcategory \(\mathcal X\).
For an ideal \(\mathcal J\) of \(\mathcal A\), define
\({}^{\perp}\mathcal J\) to be the ideal of morphisms \(i\) such that
\(\operatorname{Ext}(i,j)=0\) for every \(j\in\mathcal J\), and define
\(\mathcal I^{\perp}\) dually. More explicitly, for morphisms
\(i:A\to B\) and \(j:X\to Y\), $\operatorname{Ext}(i,j)=j_*i^*: \operatorname{Ext}(B,X)\longrightarrow\operatorname{Ext}(A,Y)$ is the homomorphism obtained by pullback along \(i\) and pushout along
\(j\).  
A pair \((\mathcal I,\mathcal J)\) is an ideal cotorsion pair if \(\mathcal I={} ^{\perp}\mathcal J\) and
\(\mathcal J=\mathcal I^{\perp}\).  
It is complete when \(\mathcal I\) is special precovering and \(\mathcal J\) is special preenveloping in the sense of ideal approximation theory \citep{FGHT2013}.

If both ideals are object ideals, put \(\mathcal F=\Ob(\mathcal I)\) and \(\mathcal C=\Ob(\mathcal J)\).
Then their objects form a cotorsion pair $(\mathcal{F},\mathcal{C})$.
The remaining question is whether completeness also descends:
\[
 (\mathcal I,\mathcal J)\text{ complete and object ideals}
 \quad\xRightarrow{?}\quad
 (\Ob(\mathcal I),\Ob(\mathcal J))\text{ complete}.       \tag{1}
\]
This question  was raised by Fu et al. \citep[Question 29]{FGHT2013}.  
A positive answer is known when the exact category has enough projective and injective objects and the right ideal is enveloping \citep{SWZ2024}, and also in the Krull–Schmidt setting \citep{ZZ2026}. 
Although positive answers are known under additional hypotheses, the question in arbitrary exact categories has remained open.  
We answer it negatively.

\begin{theorem}\label{thm:main}
Let \(\mathcal B\) be an abelian category with enough projectives and let
\(\lambda\) be an integer-valued function on the objects of \(\mathcal B\),
invariant under isomorphisms and additive on short exact sequences. Suppose
\[
 \lambda(P)=0\quad(P\in\operatorname{Proj}(\mathcal B)),
 \qquad \lambda(S)>0
\]
for some \(S\in\mathcal B\). Let
\(\mathcal A=\{M\in\mathcal B\mid\lambda(M)\geq0\}\), endowed with the
exact structure induced from \(\mathcal B\), and put
\(\mathcal P=\operatorname{Proj}(\mathcal A)\). Then
\[
 (\mathcal I(\mathcal P),\mathcal I(\mathcal A))
\]
is a complete ideal cotorsion pair of object ideals, whereas
\((\Ob(\mathcal I(\mathcal P)),\Ob(\mathcal I(\mathcal A)))
=(\mathcal P,\mathcal A)\) is not complete.  Moreover,
\(\mathcal A\) is not weakly idempotent complete.
\end{theorem}

\begin{proof}
Exact additivity makes \(\mathcal A\) additive and extension-closed.
Every projective object of \(\mathcal B\) belongs to \(\mathcal A\) and
is projective in \(\mathcal A\). Conversely, let
\(T\in\operatorname{Proj}(\mathcal A)\) and \(Y\in\mathcal B\).
Choose \(n\geq0\) such that \(\lambda(Y)+n\lambda(S)\geq0\), and put
\(Y^+=Y\oplus S^{\oplus n}\). Then \(Y^+\in\mathcal A\). Since
\(\mathcal A\) is full and extension-closed,
\[
0=\operatorname{Ext}^1_{\mathcal A}(T,Y^+)
 \cong\operatorname{Ext}^1_{\mathcal B}(T,Y)
 \oplus\operatorname{Ext}^1_{\mathcal B}(T,S)^{\oplus n}.
\]
Thus \(T\) is projective in \(\mathcal B\), and
\(\mathcal P=\operatorname{Proj}(\mathcal B)\).

Set
\[
 (\mathcal I,\mathcal J)
 :=(\mathcal I(\mathcal P),\mathcal I(\mathcal A)).
\]
Here \(\mathcal J\) is the ideal of all morphisms of \(\mathcal A\),
because every morphism factors through its domain. Every morphism in
\(\mathcal I\) factors through a projective object, hence is left
Ext-orthogonal to every morphism. Therefore
\(\mathcal I^{\perp}=\mathcal J\).

Let \(f:U\to V\) belong to \({}^{\perp}\mathcal J\), and let
\(Y\in\mathcal B\). Choose \(n\) as above and put
\(Y^+=Y\oplus S^{\oplus n}\in\mathcal A\). Since
\(1_{Y^+}\in\mathcal J\), we have
\(\operatorname{Ext}_{\mathcal A}(f,1_{Y^+})=0\). The induced exact
structure and the direct-sum decomposition of \(Y^+\) give
\[
 \operatorname{Ext}^1_{\mathcal B}(f,1_Y)=0.             \tag{2}
\]
Choose an exact sequence in \(\mathcal B\)
\[
0\longrightarrow L_V\longrightarrow P_V
 \xrightarrow{\pi_V}V\longrightarrow0,
\]
where \(P_V\) is projective. Taking \(Y=L_V\) in (2) shows that its
pullback along \(f\) splits. Equivalently, \(f\) lifts through
\(\pi_V\), so it factors through \(P_V\in\mathcal P\). Hence
\[
 {}^{\perp}\mathcal J=\mathcal I,
 \qquad \mathcal J=\mathcal I^{\perp}.                   \tag{3}
\]

It remains to prove completeness. Fix \(M\in\mathcal A\). Choose in
\(\mathcal B\) a projective epimorphism
\(\pi_M:P_M\twoheadrightarrow M\), put \(K_M=\ker\pi_M\), and define
\[
 p_M=(\pi_M,0):P_M\oplus M\longrightarrow M.
\]
The kernel of \(p_M\) is \(K_M\oplus M\), and additivity gives
\[
 \lambda(K_M\oplus M)
 =\lambda(K_M)+\lambda(M)=\lambda(P_M)=0.
\]
Thus
\[
 K_M\oplus M\longrightarrow P_M\oplus M
 \xrightarrow{p_M}M
\]
is a conflation in \(\mathcal A\). The morphism \(p_M\) belongs to
\(\mathcal I=\mathcal I(\mathcal P)\), and every morphism in
\(\mathcal I\) with codomain \(M\) factors through \(p_M\). Moreover,
the displayed conflation is its own pushout along
\(1_{K_M\oplus M}\in\mathcal J=\mathcal I^{\perp}\). Hence \(p_M\)
is a special \(\mathcal I\)-precover of \(M\).

The identity \(1_M\) is a \(\mathcal J\)-preenvelope of \(M\). The split
conflation
\[
 M\xrightarrow{1_M}M\longrightarrow0
\]
is its own pullback along
\(1_0=0\in\mathcal I={}^{\perp}\mathcal J\). Hence \(1_M\) is a
special \(\mathcal J\)-preenvelope. Therefore
\((\mathcal I,\mathcal J)\) is complete.

Since projective objects are closed under retracts,
\(\Ob(\mathcal I)=\mathcal P\); clearly
\(\Ob(\mathcal J)=\mathcal A\). If \((\mathcal P,\mathcal A)\) were
complete, \(S\) would admit a conflation
\[
 L\longrightarrow P'\longrightarrow S
\]
in \(\mathcal A\), with \(P'\in\mathcal P\) and \(L\in\mathcal A\).
Then
\[
 \lambda(L)=\lambda(P')-\lambda(S)=-\lambda(S)<0,
\]
a contradiction.  
Thus \((\mathcal P,\mathcal A)\) is not complete.

Finally, choose a projective epimorphism
\(\pi_S:P_S\twoheadrightarrow S\) in \(\mathcal B\). The composite
\[
 (\pi_S,0)=\pi_S\operatorname{pr}_{P_S}:
 P_S\oplus S\longrightarrow S
\]
is a deflation in \(\mathcal A\), whereas \(\pi_S\) is not, since
\(\lambda(\ker\pi_S)=-\lambda(S)<0\).
Deflation cancellation fails, so \(\mathcal A\) is not weakly idempotent complete.
\end{proof}

\begin{example}
    Let \(k\) be a field and
\[
 Q:\begin{tikzcd}
1 \arrow[rr, "\alpha", bend left] &  & 2 \arrow[ll, "\beta", bend left]
\end{tikzcd},
 \qquad \Lambda=kQ/(\alpha\beta,\beta\alpha).
\]
An object \(M\in\mathcal B=\operatorname{mod}\text{-}\Lambda\) is a finite-dimensional representation \(M=(M_1,M_2;M_\alpha,M_\beta)\) of
\(Q\).  
Define
\[
\ell(M)=\dim_k M_1-\dim_k M_2.
\]
Both indecomposable projectives have dimension vector \((1,1)\), while the simple \(S_1\) has \(\ell(S_1)=1\).  
Theorem~\ref{thm:main} therefore applies to \(\mathcal A=\{M\mid\ell(M)\ge0\}\).

The shape of the example explains the construction.  
The second arrow \(\beta\) makes both projective dimension vectors lie in the kernel of \(d_1-d_2\); for the one-way quiver this functional does not annihilate
all projectives.  
The relations kill the two length-two cycles, make the algebra finite-dimensional, and leave both projectives with one top and one radical composition factor.  
The essential ingredient is therefore not the chosen quiver but a nonzero exact-additive function annihilating projective objects.

The ideal approximation repairs a forbidden kernel by adjoining the target: \(K_M\) may have negative defect, while \(K_M\oplus M\) has defect
zero.  
This also explains the failure of weak idempotent completeness.
Indeed, if \(S_2\) is the other simple, then \(S_1,S_1\oplus S_2\in\mathcal A\) but \(S_2\notin\mathcal A\); the retraction \(S_1\oplus S_2\to S_1\) has no kernel in \(\mathcal A\).
\end{example}

The counterexample consequently settles the question only for arbitrary exact categories.  
It does not decide whether completeness descends in every weakly idempotent complete exact category or every abelian category.

\bibliographystyle{plain}
\bibliography{references}

@article{FGHT2013,
  author={Fu, X. H. and Guil Asensio, P. A. and Herzog, I. and Torrecillas, B.},
  title={Ideal approximation theory},
  journal={Advances in Mathematics},
  volume={244},
  pages={750--790},
  year={2013},
  doi={10.1016/j.aim.2013.05.020}
}

@incollection{Salce1979,
  author={Salce, L.},
  title={Cotorsion theories for abelian groups},
  booktitle={Symposia Mathematica},
  volume={23},
  pages={11--32},
  year={1979},
  publisher={Academic Press},
  address={London}
}

@misc{SWZ2024,
  author={Sun, D. and Wang, Q. and Zhu, H.},
  title={Cotorsion pairs and {Enochs Conjecture} for object ideals},
  year={2024},
  eprint={2412.05519},
  archivePrefix={arXiv},
  note={arXiv:2412.05519},
  doi={10.48550/arXiv.2412.05519}
}

@misc{ZZ2026,
  author={Zhang, Y. and Zhou, P.},
  title={Ideal {$n$}-cotorsion pairs in {Frobenius} extriangulated categories},
  year={2026},
  eprint={2606.29728},
  archivePrefix={arXiv},
  note={arXiv:2606.29728},
  doi={10.48550/arXiv.2606.29728}
}

\vspace{4mm}
\small
\noindent\textbf{Qikai Wang}\\
School of Engineering, Westlake University, Hangzhou 310014, China\\
E-mail: wangqikai@westlake.edu.cn\\[1mm]
\textbf{Yuxiao Wang}\\
School of Mathematical Science, Zhejiang University of Technology, Hangzhou 310023, China\\
E-mail: 524517300@qq.com\\[1mm]
\textbf{Haiyan Zhu}\\
School of Mathematical Science, Zhejiang University of Technology, Hangzhou 310023, China\\
E-mail: hyzhu@zjut.edu.cn\\[1mm]
\end{document}